\documentclass[
final
]{dmtcs-episciences}

\usepackage[utf8]{inputenc}
\usepackage{subfigure}

\usepackage[numbers]{natbib}

\author[Aboulker et al.]{Pierre Aboulker\affiliationmark{1}\thanks{Funded by group Casino/ENS Chair on Algorithmics and Machine Learning, ELIT (ANR-20-CE48-0008-01).}
  \and Guillaume Aubian\affiliationmark{2}\thanks{Funded by by project 22-17398S (Flows and cycles in graphs on surfaces) of Czech Science Foundation.}
  \and Raul Lopes\affiliationmark{1,3}$^{,*}$}
\title{Finding forest-orderings of tournaments is {\sf NP}-complete}
\affiliation{
  DIENS, \'Ecole normale sup\'erieure, CNRS, PSL University, Paris, France\\
  Charles University, Prague, Czech Republic\\
  LIRMM, Universit\'e de Montpellier, Montpellier, France}
\usepackage{xspace}
\usepackage{url}

\usepackage{longtable}
\usepackage{amsmath,amsthm,amssymb,amsthm, stmaryrd}
\usepackage{enumerate}
\usepackage{multirow}
\usepackage{algorithm}
\usepackage{algpseudocode}
\usepackage[UKenglish]{babel}
\algnewcommand\algorithmicinput{\textbf{Input:}}
\algnewcommand\algorithmicoutput{\textbf{Output:}}
\algnewcommand\Input{\item[\algorithmicinput]}
\algnewcommand\Output{\item[\algorithmicoutput]}

\usepackage{array}
\usepackage{booktabs}

\usepackage{hyperref}
\addto\extrasUKenglish{%
}

\DeclareMathOperator{\dic}{\overrightarrow{\chi}}
\DeclareMathOperator{\diomega}{\overrightarrow{\omega}}

\newcommand{\Left}{{\sf L}\xspace}

\newcommand{\Right}{{\sf R}\xspace}

\newcommand{\NP}{{\sf NP}\xspace}

\newcommand{\occur}{{\sf occ}\xspace}

\newcommand{\ov}{\overline}

\newcommand{\Ra}{\Rightarrow}
\newcommand{\ra}{\rightarrow}

\newcommand{\mc}{\mathcal}

\newenvironment{subproof}{\par\noindent {\textbf{Proof of the claim}}.\ }{\hfill$\lozenge$\par\vspace{11pt}}

\newcommand{\T}{\xspace T^\prec \xspace}
\newcommand{\I}{\xspace \mathcal{I} \xspace}

\usepackage{tikz}

\usetikzlibrary{mindmap,positioning}
\usetikzlibrary{graphs}
\usetikzlibrary[graphs]
\usetikzlibrary{patterns}
\usetikzlibrary{arrows,automata,positioning}
\usetikzlibrary{decorations.pathreplacing}
\usetikzlibrary{decorations}
\usetikzlibrary{decorations.pathmorphing}
\usetikzlibrary{shapes.geometric}
\usetikzlibrary{quotes}
\usetikzlibrary{shapes.callouts}
\usetikzlibrary{arrows}
\usetikzlibrary{arrows.meta}
\usetikzlibrary{calc}
\usetikzlibrary{decorations}
\usetikzlibrary{fit}
\usetikzlibrary{decorations.markings}

\usepackage{pythonhighlight}

\tikzset{faded/.style={gray,very thin}}
\tikzset{very faded/.style={gray!50}}
\tikzset{vertex/.style={draw,circle,minimum size=5pt,inner sep=0pt}}
\tikzset{novertex/.style={circle,minimum size=5pt,inner sep=0pt}}
\tikzset{blackvertex/.style={draw,circle,minimum size=5pt,inner sep=0pt, fill=black}}
\tikzset{redvertex/.style={draw,circle,minimum size=5pt,inner sep=0pt, fill=red}}
\tikzset{redvertexfaded/.style={draw,circle,faded,minimum size=5pt,inner sep=0pt, fill=red!50}}
\tikzset{greenvertex/.style={draw,circle,minimum size=5pt,inner sep=0pt, fill=green}}
\tikzset{greenvertexfaded/.style={draw,circle,faded,minimum size=5pt,inner sep=0pt, fill=green!50}}
\tikzset{bluevertex/.style={draw,circle,minimum size=5pt,inner sep=0pt, fill=blue}}
\tikzset{bluevertexfaded/.style={draw,circle,faded,minimum size=5pt,inner sep=0pt, fill=blue!50}}
\tikzset{yellowvertex/.style={draw,circle,minimum size=5pt,inner sep=0pt, fill=yellow}}
\tikzset{yellowvertexfaded/.style={draw,circle,faded,minimum size=5pt,inner sep=0pt, fill=yellow!50}}

\tikzset{arrow/.style={-{Latex[scale=1]}, shorten >= 0pt}} %% HERE HERE HERE

\tikzset{double arrow/.style={-{Latex[scale=1]}, line width = 2pt, shorten >= 0pt}}

\tikzset{edge/.style = {->,> = latex'}}

\tikzset{snake it/.style={decorate, decoration=snake}}

\tikzset{->-/.style={decoration={
  markings,
  mark=at position .5 with {\arrow{>}}},postaction={decorate}}}

\usepackage[nothm]{thmbox}
\usepackage{amsthm}
\usepackage{thmtools}
\usetikzlibrary{patterns}
\usetikzlibrary{calc}

\declaretheorem[parent=section]{theorem}

\declaretheorem[numberlike=theorem]{conjecture}
\declaretheorem[numberlike=theorem]{problem}

\newtheorem{claim}{Claim}[theorem]

\declaretheorem[numberlike=theorem]{lemma}

\begin{document}
% This is only used if you are compiling for a volume before vol 25
% \publicationdetails{VOL}{2015}{ISS}{NUM}{SUBM}
% This is the new form of collecting the data, starting with vol 25
\publicationdata{vol. 28:2}{2026}{23}{10.46298/dmtcs.14281}{2024-09-17; 2024-09-17; 2026-02-05}{2026-03-21}
\maketitle

\begin{abstract}
   Given a class of (undirected) graphs $\mathcal{C}$, we say that a Feedback Arc Set (FAS for short) $F$ is a $\mathcal{C}$-FAS if the graph induced by the edges of $F$ (forgetting their orientations) belongs to $\mathcal{C}$. We show that deciding if a tournament has a $\mathcal{C}$-FAS is \NP-complete when $\mathcal{C}$ is the class of all forests. We are motivated by connections between $\mathcal{C}$-\textsc{FAS} and structural parameters of tournaments, such as the dichromatic number, the clique number of tournaments, and the strong Erdős-Hajnal property.
\end{abstract}

\section{Introduction}

Given a tournament $T$, a \emph{Feedback Arc Set} (\emph{FAS} for short) is a set of arcs $F$ of $T$ such that $T \setminus F$ is acyclic.
Feedback arc sets in tournaments are well studied from the combinatorial \cite{10.1007/978-3-642-02927-1_32,Dwork2002RankAR,Kemeny1959,Kemeny_Snell_1962,Raman_Saurabh_2006,Speckenmeyer_1989}, statistical \cite{Kenyon_Mathieu_Schudy_2007}, and algorithmic \cite{ABUKHZAM2010524,10.1145/1411509.1411513,BFGPPST11,NIPS1997_e11943a6,CTY07,Fomin_Lokshtanov_Raman_Saurabh_2010,Seshu_Reed_1961,Slater_1961} points of view.

Given a class of (undirected) graph $\mathcal C$, we say that a FAS is a $\mathcal C$-FAS if the graph induced by the edges of $F$ (forgetting their orientations) belongs to $\mathcal C$, and we call  \emph{$\mathcal C$-\textsc{FAS} Problem} the associated decision problem, that is deciding if a tournament has a $\mathcal C$-FAS.
The goal of this paper is to prove the following:
\begin{theorem}\label{thm:main}
    The $\mathcal C$-\textsc{FAS} Problem is NP-complete when $\mathcal C$ is the set of forests.
\end{theorem}

\paragraph{Motivations and related works}
Let us motivate the $\mathcal C$-\textsc{FAS} Problem by showing its link with several studied problems in tournament theory.
First, let us abuse notation and refer to a FAS $F$ of a tournament as the undirected graph induced by the edges of $F$ (forgetting their orientations).

Given a tournament $T$, and a total order $\prec$ on $V(T)$, we denote by $T^{\prec}$ the (undirected) graph   with vertex set $V(T)$ and edge $uv$ if $u\prec v$ and $vu\in A(D)$. We call it the \emph{backedge graph} of $T$ with respect to $\prec$.
Observe that $F$ is a FAS of $T$ if and only if there is an ordering $\prec$ such that $F=T^{\prec}$.

Given a tournament $T$, we denote by $\dic(T)$ its {\em dichromatic number}, that is the minimum integer $k$ such that the set of vertices of $T$ can be partitioned into $k$ transitive tournaments. The dichromatic number of tournaments has recently been a centre of interest~\cite{BCCFLSST13,HTW19,KN24}, in particular because of the bridges it creates between tournament theory and (undirected) graph theory.
For example, a tournament version~\cite{APS01} of the  Erd\H{o}s-Hajnal Conjecture~\cite{EH77} has been studied quite a lot~\cite{BCCZ22,NSS25,ZAYAT2024113920,GZ23}, as well as a tournament version~\cite{NSS24,KFN24} of the El-Zahar-Erd\H{o}s Conjecture~\cite{EE85}. A better understanding of tournaments could thus lead to a better understanding of graphs.

Let $T$ be a tournament and $\prec$ an ordering of its vertices. It is straightforward that every independent set of $T^{\prec}$ induces a transitive subtournament of $T$. As a consequence, we have that $\dic(T) \le \chi(T^{\prec})$.
Conversely, by taking an ordering built from a $\dic(T)$-dicolouring,
that is taking colour classes one after the other, and ordering each colour class in a topological ordering, we get that:
\begin{equation}\label{eq:dic_def}
    \dic(T) = \min \, \big\{ \chi(T^{\prec}) : \mbox{$\prec$ is a total order of $V(D)$} \big\}
\end{equation}
Now, it is straightforward that a tournament $T$ is $k$-dicolourable if and only it contains a $k$-colourable FAS. Hence, deciding if a tournament is $k$-dicolourable is equivalent with deciding if a tournament has a $k$-colourable FAS.
The former has been shown to be NP-complete for every $k \geq 2$~\cite{DFJM04}.
In addition, Chen et al.~\cite{chen2007min} showed that it is NP-complete to decide if a tournament admits a bipartite backedge graph.
Our result can be seen as a strengthening of this result.
% Before, it was like this:
%This problem has been proved to be NP-complete for every $k \geq 2$~\cite{DFJM04}.
%In particular, the $\mathcal C$-\textsc{FAS} problem is NP-complete when $\mathcal C$ is the class of bipartite graphs. Our result can be seen as a strengthening of this result.

Inspired by~\eqref{eq:dic_def}, the clique number $\diomega(T)$ of a tournament $T$ has recently been defined~\cite{AACL23} (see also~\cite{NGUYEN2025146}) as follows:
% NEWSTUFF
\begin{equation*}\label{eq:diomega_def}
    \diomega(T) = \min \, \big\{ \omega(T^{\prec}) : \mbox{$\prec$ is a total order of $V(T)$} \big\}
\end{equation*}
Hence, $\diomega(T) \leq k$ if and only if $T$ has a FAS $F$ inducing a graph with clique number at most $k$. Denoting  $\mathcal C_k$ the class of graphs with clique number at most $k$, deciding if a tournament has clique number at most $k$ is equivalent with the $\mathcal C_k$-\textsc{FAS} problem. The problem of computing the clique number of a tournament has been mentioned in~\cite{SS20} and in~\cite{AACL23}. It has very recently been proved~\cite{A24} to be NP-complete for $k \geq 3$ and is still open for $k=2$.

A tournament $H$ has  the \emph{strong Erd\H{o}s-Hajnal property} if there exists $\epsilon_H >0$ such that every $H$-free tournament $T$ contains two disjoint subset of vertices $A$ and $B$, such that $|A|,|B| \geq \epsilon_H |V(T)|$ and all arcs between $A$ and $B$ are oriented from $A$ to $B$. Chudnovsky et al.~\cite{CSSS24} conjectured that a tournament $H$ has the strong Erd\H{o}s-Hajnal property if and only if $H$ admits a FAS inducing a forest, showing that these tournaments might play an important role in tournament theory.
\medskip

\paragraph{Structure of the paper}
Definitions and some preliminaries results are given in \autoref{sec:def}.
The proof of \autoref{thm:main} is a reduction from $3$-\textsc{SAT}.
Given an instance $\mc I$ of $3$-\textsc{SAT}, we construct a tournament $T_{\I}$ such that $T_{\I}$ has polynomial size (in the size of $\I$)  and has a forest-ordering if and only if $\I$ is a satisfiable  instance.
We explain how to construct $T_{\I}$ in \autoref{subsec:construction}. We prove that if $T_{\I}$ has a forest-ordering, then $\I$ is a satisfiable instance in \autoref{subsec:forest-ordering_implies_True_Assignment}, and finally, we prove that if $\I$ is a satisfiable  instance, then $T_{\I}$ has a forest-ordering in \autoref{subsec:assignment_implies_forest_ordering}.

\section{Definitions and preliminaries}\label{sec:def}

We refer the reader to~\cite{digraphs_book} for snoteAtandard definitions.

An \emph{orientation} of a graph $G$ is the digraph obtained by assigning an orientation to every edge of $G$.
A \emph{tournament} is an orientation of a complete graph.
If $D$ is a digraph, we denote by $V(D)$ and $A(D)$ the set of vertices and arcs of $D$, respectively.
The \emph{transitive tournament} on $n$ vertices is the unique acyclic tournament on $n$ vertices.

An \emph{ordered tournament} is a pair $(T, \prec)$ where $T$ is a tournament and $\prec$ is a total ordering of $V(T)$.
If $uv \in A(T)$ with $v \prec u$ we say that $uv$ is a \emph{back-arc} of $(T, \prec)$ and a \emph{forward-arc} otherwise.
We denote by $T^{\prec}$ the (undirected) graph  with vertex set $V(T)$ and edge set $\{uv \mid v \prec u \text{  and } uv\in A(T)\}$, and call it the \emph{backedge graph} of $T$ with respect to $\prec$.
A \emph{topological-ordering} of $T$ is an ordering of $V(T)$ such that every arc is a forward arc. Note that only transitive tournaments admit a topological ordering.
A \emph{(tree) forest-ordering} of $T$ is an ordering $\prec$ of $V(T)$ such that $T^{\prec}$ is a (tree) forest.

For an ordered tournament $(T, \prec)$ and $X, Y \subseteq V(T)$, we write $X \prec Y$ to say that every vertex of $X$ precedes every vertex of $Y$.
If $X = \{v\}$ we drop the brackets from the notation and simply write $v \prec Y$ or $Y \prec v$.
If $V(T') \prec v$ or $v \prec V(T')$ for a subtournament $T'$ of $T$, we shorten the notation to $T' \prec v$ and $v \prec T'$.

If $x,y$ are vertices of a digraph $D$, we may also write $x \ra y$ to say that there is an arc from $x$ to $y$.
Given two disjoint sets of vertices $X, Y$ of $D$, we write $X \Rightarrow Y$ to say that for every $x \in X$ and for every $y \in Y$, $xy \in A(D)$.
When $X=\{x\}$ or $Y = \{y\}$, we simply write $x \Rightarrow Y$ or $X \Ra \{y\}$, respectively.
We also use the symbol $\Ra$ to denote a composition operation on tournaments: for two tournaments $T_1$ and $T_2$, $T_1\Ra T_2$ is the digraph obtained from the disjoint union of $T_1$ and $T_2$ by adding all arcs from $V(T_1)$ to $V(T_2)$.

For a positive integer $k$, we denote by $[k]$ the set $\{1, \ldots, k\}$.

\medskip

\paragraph{Preliminaries results}
As observed in the introduction, an ordering $\prec$ of a tournament $T$ is a forest-ordering if and only if $T^{\prec}$ is a forest. Throughout the paper, we look at $T^{\prec}$ to decide if $\prec$ is a forest ordering or not.
\medskip

The following simple technical lemma is used several times to prove properties of forest-orderings.

\begin{lemma}\label{lem:tool}
    Let $(T, \prec)$ be an ordered tournament. Let $a,b \in V(T)$, $X \subseteq V(T) \setminus \{a,b\}$ and assume that  $a \Ra X \Ra b$. If $b \prec a$, then each vertex of $X$ is adjacent to $a$ or $b$ in $T^{\prec}$.
\end{lemma}

\begin{proof}
    Let $x \in X$. If $x \prec a$, then $ax \in E(T^{\prec})$, and if $a \prec x$, then $b \prec x$ and $bx \in E(T^{\prec})$.
\end{proof}

Now, we describe a tournament that admits a unique forest-ordering which, in particular, is a tree-ordering.
This statement was verified by a program implementing the code described in \autoref{section:magical_code}.
\begin{lemma}\label{lem:unique_tree_ordering}
    The tournament depicted in \autoref{figure:magical-tournament}  has a unique forest-ordering, which is the one shown in \autoref{figure:magical-tournament}. This ordering is tree-ordering.
\end{lemma}

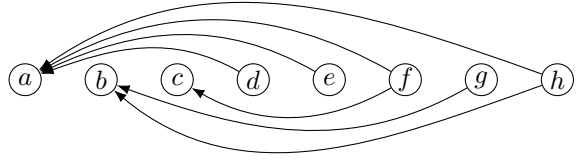
\begin{figure}[h]
    \centering
    \vspace{-1cm}
    \begin{tikzpicture}[scale=1]
        \node[vertex, minimum size=12pt] (a) at (1,0) {$a$};
        \node[vertex, minimum size=12pt] (b) at (2,0) {$b$};
        \node[vertex, minimum size=12pt] (c) at (3,0) {$c$};
        \node[vertex, minimum size=12pt] (d) at (4,0) {$d$};
        \node[vertex, minimum size=12pt] (e) at (5,0) {$e$};
        \node[vertex, minimum size=12pt] (f) at (6,0) {$f$};
        \node[vertex, minimum size=12pt] (g) at (7,0) {$g$};
        \node[vertex, minimum size=12pt] (h) at (8,0) {$h$};

        \draw[arrow] (d) to [out = 150, in = 20] (a);
        \draw[arrow] (e) to [out = 150, in = 25] (a);
        \draw[arrow] (f) to [out = 150, in = 30] (a);
        \draw[arrow] (h) to [out = 160, in = 35] (a);
        \draw[arrow] (g) to [out = 150, in = 20] (b);
        \draw[arrow] (h) to [out = 145, in = 35] (b);
        \draw[arrow, bend right=30] (f) to (c);

        \foreach \i/\j in {a/b, b/c, c/d, d/e, e/f, f/g, g/h}{
                \draw[arrow, very faded] (\i) to (\j);
            }
        \foreach \i in {c,g}{
                \draw[arrow, very faded] (a) to [bend right=30] (\i);
            }
        \foreach \i in {d,e,f}{
                \draw[arrow, very faded] (b) to [bend right=30] (\i);
            }
        \foreach \i in {e,g,h}{
                \draw[arrow, very faded] (c) to [bend right=30] (\i);
            }
        \foreach \i in {f,g,h}{
                \draw[arrow, very faded] (d) to [bend right=30] (\i);
            }
        \foreach \i in {g,h}{
                \draw[arrow, very faded] (e) to [bend right=30] (\i);
            }
        \foreach \i in {h}{
                \draw[arrow, very faded] (f) to [bend right=30] (\i);
            }

    \end{tikzpicture}
    \caption{A tournament on eight vertices admitting a unique forest-ordering (left to right on the figure). Forward arcs are faded.}
    \label{figure:magical-tournament}
\end{figure}

Let $n \geq 1$ and let $T_1$ and $T_2$ be vertex-disjoint transitive tournaments with topological orderings $(u_1, \ldots, u_n)$ and $(v_1, \ldots, v_n)$, respectively.
Let $\prec$ be an ordering of $V(T_1) \cup V(T_2)$ where $u_1 \prec \ldots \prec u_n \prec v_1 \prec \ldots \prec v_n$.
We say that we \emph{add a back-arc matching from $T_2$ to $T_1$ with respect to $\prec$} (or simply \emph{add a back-arc matching from $T_2$ to $T_1$} when $\prec$ is clear from the context) if we add the following set of arcs between $T_1$ and $T_2$:
$$
\{\, v_i u_i \mid i \in [n] \,\} \;\cup\; \{\, u_j v_i \mid i,j \in [n]\ \text{and}\ i \neq j \,\}.
$$
Thus, $T^{\prec}$ consists of a perfect matching from $V(T_2)$ to $V(T_1)$. The \emph{back-arcs} of the back-arc matching are arcs $\{v_i u_i \mid i \in [n]\}$.

See \autoref{figure:backedge-matching} for an example of a back-arc matching.

\begin{figure}[h]
    \centering
    \begin{tikzpicture}[yscale=1]
        \foreach \i in {1,2,3} {
                \node[blackvertex, scale = .7] (u\i) at (\i-1,0) {};
                \node[blackvertex, scale = .7] (v\i) at (3+\i-1,0) {};
            }

        \node[rectangle, draw, fit=(u1)(u3), label = 180:$T_1$] {};
        \node[rectangle, draw, fit=(v1)(v3), label = 0:$T_2$] {};

        \foreach \i in {1,2,3} {
                \draw[arrow, very thick] (v\i) to [bend right = 30] (u\i);
            }

        \draw[arrow, faded] (u1) to [bend right = 30] (v2);
        \draw[arrow, faded] (u1) to [bend right = 30] (v3);
        \draw[arrow, faded] (u2) to [bend right = 30] (v1);
        \draw[arrow, faded] (u2) to [bend right = 30] (v3);
        \draw[arrow, faded] (u3) to [bend right = 30] (v1);
        \draw[arrow, faded] (u3) to [bend right = 30] (v2);

    \end{tikzpicture}
    \caption{A back-arc matching from a transitive tournament $T_2$ to a transitive tournament $T_1$, both on three vertices, and whose topological orderings are left-to-right. The thick arcs are precisely the arcs of $T^{\prec}$, and they form a matching. Arcs inside $T_1$ and $T_2$ are omitted.}
    \label{figure:backedge-matching}
\end{figure}
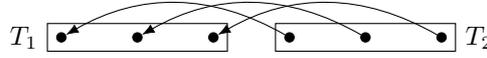

%--------------------------------------------------------

\section{The reduction}\label{sec:proof}

\subsection{The construction of $T_{\I}$}\label{subsec:construction}

Let $\I$ be an instance  of  $3$-\textsc{SAT} with $n$ variables $x_1, \dots, x_n$ and $k$ clauses $C_1, \dots, C_k$. 
%We consider an arbitrary ordering $x_1, \ldots, x_n$ of the variables and $C_1, \ldots, C_k$ of the clauses.
We also consider that for  $i \in [k]$, each clause $C_i$ is formed by an ordered trio of literals.
This naturally implies an ordering $\ell_1$, $\ell_2$, $\ell_3$, $\ldots$, $\ell_{3k-2}$, $\ell_{3k-1}$, $\ell_{3k}$ of the literals as they occur in the clauses, i.e. for $j \in [k]$ we have $C_j = (\ell_{3j-2} \vee \ell_{3j-1} \vee \ell_{3j})$.

This subsection is dedicated to the construction of the tournament $T_{\I}$ from $\mathcal{I}$.
We construct an ordered tournament $(T_{\I}, \prec^*)$ where the order $\prec^*$ is used to simplify the description of $T_{\I}$, as many arcs of $T_{\I}$ are forward-arcs with respect to $\prec^*$. The order $\prec^*$ is also used in \autoref{subsec:assignment_implies_forest_ordering} where, given an assignment of  $\I$, we need to deduce a forest-ordering  of $T_{\I}$.
The forest-ordering is found by slightly perturbing $\prec^*$.

The construction of $(T_{\I}, \prec^*)$ is done in two steps.
We first describe the construction of the ordered tournament $(T_{B}, \prec^*)$, called the \emph{base tournament}, to which gadgets are added later to get $(T_{\I}, \prec^*)$.

% 
%----------------------------------------
\subsubsection{The base tournament $(T_{B}, \prec^*)$}\label{sec:base-tournament}

To build $(T_B, \prec^*)$, we start with vertices $v_1, \ov{v}_1, \dots, v_n, \ov{v}_n, \ell_1, \dots, \ell_{3k}$ in this order.  We set $\mathcal V$ $= \{v_1,$ $\ov{v}_1,$ $\ldots,$ $v_n,$ $\ov{v}_n\}$ and $\mathcal L = \{\ell_1, \dots, \ell_{3k}\}$.
For every $x \in \mathcal V \cup \mathcal L$, we add a set of vertices $M_x$\footnote{$M$ stands for Magical, which is the name we gave to the tournament that has a unique forest-ordering} right after $x$ in the ordering, such that $T_{B}[M_x]$ is the tournament depicted in \autoref{figure:magical-tournament}, and all arcs from $x$ to $M_x$.
All non-defined arcs at this point are added as forward-arcs with respect to $\prec^*$.
Thus,
\begin{itemize}
    \item $x \Ra M_x$ for all $x\in \mathcal{V} \cup \mathcal{L}$, and 
    \item $\prec^*$ restricted to $M_x$ is the unique forest-ordering of $M_x$.
    \item all non-defined arcs as forward-arcs.
\end{itemize}

This terminates the construction of $(T_B, \prec^*)$. Each $v_i$, $\ov{v}_i$ represents respectively the variable $x_i$ and its negation, and $\ell_i$ represent the $i^{th}$ literal.

For every $x \in \mathcal V \cup \mathcal L$, the smallest vertex of $T_{B}[M_x]$ (vertex $a$ in \autoref{figure:magical-tournament}) plays a very important role (as we explain below), and is called $\ell_x$.
For every $x \in \mathcal V \cup \mathcal L$, we call $B_x$ the tournament $T_{B}[\{x \cup M_x\}]$ and $\prec_x$ the restriction of $\prec^*$  to $B_x$.
% NEWSTUFF
The ordered tournament $(B_x, \prec_x)$ is called a \emph{block}.

%-----------------------------------------

\subsubsection{The gadgets}\label{sec:gadgets}
% NEWSTUFF
We now add some gadgets to $(T_{B}, \prec^*$), on top of each block $B_x$.
For each $w \in \mathcal{V} \cup \mathcal{L}$, we add to $(T_B, \prec^*)$ a copy of a transitive tournament that is then partitioned in a way that depends on whether $w$ is associated with a variable or with a literal.
For each variable $x_i$ of the given instance $\mathcal{I}$ of $3$-\textsc{SAT}, denote by $\occur(x_i)$ the number of occurrences of the literal $x_i$ and by $\occur(\overline{x}_i)$ the number of occurrences of the literal $\overline{x}_i$.

We remind the reader that every $v \in \mathcal{V}$ is of the form $v_i$ or $\ov{v}_i$. 
If $v = v_i$ for some $i\in [n]$, we denote by $\ov{v}$ the vertex $\ov{v}_i$.
If $v = \ov{v}_i$, we denote by $\ov{v}$ the vertex $v_i$.
Now, for $v \in \mathcal{V}$ we define the gadget $G_v$ associated with $v$ to be the transitive tournament on $7 + 7\occur(\ov{v})$ vertices.
We first partition the vertices of $G_v$ into a set $N_v$ containing the first $2+2\occur(\ov{v})$ vertices, and a set $Y_v$ containing the remaining $5+5\occur(\ov{v})$ vertices.
We then partition $N_v$ (resp. $Y_v$) into $1 + \occur(\ov{v})$ parts of size $2$ (resp. of size $5$) as follows (See \autoref{fig:big_block}):
\[
    N_v=N_v^{\ov{v}} \sqcup N_v^1 \sqcup \dots \sqcup N_v^{\occur(\ov{v})}
    \quad \text{ and } \quad
    Y_v=Y_v^{\ov{v}} \sqcup Y_v^1 \sqcup \dots \sqcup Y_v^{\occur(\ov{v})}
\]

Let us now explain how we define the gadget associated with vertices in $\mathcal L$.
Let $\ell \in \mathcal L$ and assume $\ell$ represents the variable $v_i$ or $\ov{v_i}$ and is part of the clause $C$.
The gadget $G_{\ell}$ associated with $\ell$  is a transitive tournament on $9$ vertices. We partition $G_{\ell}$ into a set $N_{\ell}$ containing the first four vertices of $G_{\ell}$ and a set $Y_{\ell}$ containing the remaining $5$ vertices.  Then, we partition $N_{\ell}$ into two sets of size $2$, $N_{\ell}^{v_i}$ and $N_{\ell}^{C}$ if $\ell$ represents $v_i$, and $N_{\ell}^{\ov{v}_i}$ and $N_{\ell}^{C}$ if $\ell$ represents $\ov{v}_i$.
See  \autoref{fig:big_block_for_literal}.

%In any case, each $w \in \mathcal{V}\cup\mathcal{L}$ is associated with a unique gadget $G_w$ consisting of a transitive tournament whose vertex set contains the sets $N_w$ and $Y_w$. The sizes and partitions of $G_w$, $N_w$, and $Y_w$ depend on whether $w$ is in $\mathcal{V}$ or $\mathcal{L}$.

%---------------------------------------------------------------------

\subsubsection{Ordering and arcs linking the gadgets and the base tournaments}\label{sec:ordering-arcs-linking-gadgets-base-tournament}

We now describe the orientation of the arcs linking the gadgets with the base tournament, as well as the way the gadgets are introduced in the ordering $\prec^*$.

First, for every $w \in \mathcal V \cup \mathcal L$, we set
\[w \prec^* N_w \prec^* \ell_w \prec^* Y_w \prec^* M_w \setminus \{\ell_w\}\]
and $\prec^*$ restricted to $G_w$ is the topological ordering of $G_w$. This defines the ordering $\prec^*$.
All arcs linking $G_w$ and $B_w$ are forward-arcs with respect to $\prec^*$, except for the arcs between $Y_w$ and $w$ that are all oriented from $Y_w$ to $w$.
See \autoref{fig:big_block} for the case when $w \in \mathcal V$ and \autoref{fig:big_block_for_literal} for the case when $w \in \mathcal L$.

\begin{figure}[h]
    \centering
    \begin{tikzpicture}
        \begin{scope}
            \def \retwidth{4}
            \def \retheight{0.6cm}
            \node[blackvertex, label = -90:{$v$}] (x) at (0,0) {};

            \begin{scope}
                \node[rectangle, draw, minimum width = \retwidth.cm, minimum height = \retheight, label = -90:{$N_v$}] (Rx) at ($(x) + (0.5 + \retwidth/2, 0)$) {};

                \foreach \i in {1/4,2/4,3/4} {
                        \draw ($(Rx) + (-\retwidth/2 + \i*\retwidth, -\retheight/2)$) -- ($(Rx) + (-\retwidth/2 + \i*\retwidth, \retheight/2)$);

                    }

                \node at ($(Rx) + (-\retwidth/2 + 1/8*\retwidth, 0)$) {$N^{\overline{v}}_{v}$};
                \node at ($(Rx) + (-\retwidth/2 + 3/8*\retwidth, 0)$) {$N^{1}_v$};
                \node at ($(Rx) + (-\retwidth/2 + 5/8*\retwidth, 0)$) {$\cdots$};
                \node at ($(Rx) + (-\retwidth/2 + 7/8*\retwidth, 0)$) {$N^{\occur(\overline{v})}_v$};
            \end{scope}

            \node[blackvertex, label = -90:{$\ell_v$}] (lx) at ($(Rx) + (0.5 + \retwidth/2, 0)$) {};

            \begin{scope}
                \node[rectangle, draw, minimum width = \retwidth.cm, minimum height = \retheight, label = -90:{$Y_v$}] (Lx) at ($(lx) + (0.5 + \retwidth/2, 0)$) {};
                \foreach \i in {1/4,2/4,3/4} {
                        \draw ($(Lx) + (-\retwidth/2 + \i*\retwidth, -\retheight/2)$) -- ($(Lx) + (-\retwidth/2 + \i*\retwidth, \retheight/2)$);

                    }
                \node at ($(Lx) + (-\retwidth/2 + 1/8*\retwidth, 0)$) {$Y^{\overline{v}}_v$};
                \node at ($(Lx) + (-\retwidth/2 + 3/8*\retwidth, 0)$) {$Y^{1}_v$};
                \node at ($(Lx) + (-\retwidth/2 + 5/8*\retwidth, 0)$) {$\cdots$};
                \node at ($(Lx) + (-\retwidth/2 + 7/8*\retwidth, 0)$) {$Y^{\occur(\overline{v})}_v$};
            \end{scope}

            \node[rectangle, draw, minimum width = 1cm, minimum height = \retheight] (Mx) at ($(Lx) + (1 + \retwidth/2, 0)$) {$M_v \setminus \ell_v$};

            \draw[double arrow, blue] (Lx) to [out=150, in = 42] (x);
            \draw[arrow] (Mx) to [bend right = 40] (lx);
        \end{scope}
    \end{tikzpicture}
    \caption{The figure represents a vertex $v \in \mathcal V$, $M_v$ and $G_v = N_v \sqcup Y_v$ ordered as in $\prec^*$. Forward-arcs are note drawn.
        The tournament induced by $M_x$ is the same as the one depicted in \autoref{figure:magical-tournament}, and thus the back-arcs in $T[M_v]$ induce a tree. A thick arc represents all arcs in that orientation.}
    \label{fig:big_block}
\end{figure}

\begin{figure}[h]
    \centering
    \begin{tikzpicture}
        \begin{scope}
            \def \retwidth{2}
            \def \retheight{0.6cm}
            \node[blackvertex, label = -90:{$z$}] (x) at (0,0) {};

            \begin{scope}
                \node[rectangle, draw, minimum width = \retwidth.cm, minimum height = \retheight, label = -90:{$N_z$}] (Rx) at ($(x) + (0.5 + \retwidth/2, 0)$) {};

                \foreach \i in {1/2} {
                        \draw ($(Rx) + (-\retwidth/2 + \i*\retwidth, -\retheight/2)$) -- ($(Rx) + (-\retwidth/2 + \i*\retwidth, \retheight/2)$);
                    }

                \node at ($(Rx) + (-\retwidth/2 + 1/4*\retwidth, 0)$) {$N^{\overline{v}_i}_z$};
                \node at ($(Rx) + (-\retwidth/2 + 3/4*\retwidth, 0)$) {$N^{C}_z$};
            \end{scope}

            \node[blackvertex, label = -90:{$\ell_z$}] (lx) at ($(Rx) + (0.5 + \retwidth/2, 0)$) {};

            \begin{scope}
                \node[rectangle, draw, minimum width = \retwidth.cm, minimum height = \retheight, label = -90:{$Y_z$}] (Lx) at ($(lx) + (0.5 + \retwidth/2, 0)$) {};
                \foreach \i in {1/6, 2/6, 3/6, 4/6, 5/6}{
                \node[blackvertex, scale = .7] (y\i) at ($(Lx) + (\i*\retwidth - \retwidth/2, 0)$ ) {};
                }

            \end{scope}

            \node[rectangle, draw, minimum width = 1cm, minimum height = \retheight] (Mx) at ($(Lx) + (1 + \retwidth/2, 0)$) {$M_z \setminus \ell_z$};

            \draw[double arrow, blue] (Lx) to [out=150, in = 42] (x);
            \draw[arrow] (Mx) to [bend right = 40] (lx);
        \end{scope}
    \end{tikzpicture}
    \caption{The figure represents a vertex $z \in \mathcal L$, $M_z$ and $G_z = N_z \sqcup Y_z$ ordered as in $\prec^*$. Forward-arcs are not drawn.
        The tournament induced by $M_z$ is the same as the one depicted in \autoref{figure:magical-tournament}, and thus the back-arcs in $T[M_z]$ induce a tree. A thick arc represents all arcs in that orientation.}
    \label{fig:big_block_for_literal}
\end{figure}
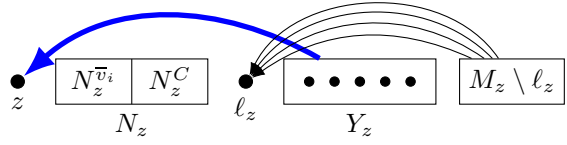

\subsubsection{Arcs linking the gadgets together}\label{section:arcs-linking-gadgets}
We now describe the orientation of the arcs between a gadget and another gadget.
Each vertex $\ell \in \mathcal{L}$ in our construction is associated with a literal $\ell_j$ of the instance $\mathcal{I}$ of $3$-\textsc{SAT}.
If $\ell_j$ is an occurrence of $x_i$ in the literals, then we say that $\ell$ is an \emph{occurrence} of $v_i$.
Otherwise, $\ell_j$ is an occurrence of $\ov{x}_i$, and we say that $\ell$ is an \emph{occurrence} of $\ov{v}_i$.
For the remainder of the text, unless stated otherwise, all the back-arc matchings are added with respect to $\prec^*$. 
%we say that $\ell$ is an \emph{occurrence} of $v$ ($\ov{v}$) if $v \in \{v_i, \ov{v}_i\}$ for some $i \in n$ and $\ell$ is associated with an occurrence of $x_i$ $(\ov{x}_i)$ in the literals.
\begin{enumerate}
    \item For every $v \in \mathcal V$ and for each occurrence $\ell \in \mathcal L$ of $\ov{v}$, if $\ell$ is the $i^{th}$ occurrence of $\ov{v}$, then we add a back-arc matching from $N_{\ell}^{\ov{v}}$ to $N_v^i$ and from $Y_{\ell}$ to $Y_v^i$. See \autoref{figure:A-gadget-A-link}.
    \item For every $v \in \{v_1, \dots, v_n\}$, we add a back-arc matching from $N_{\ov{v}}^v$ to $N^{\ov{v}}_v$, and from  $Y_{\ov{v}}^v$ to $Y^{\ov{v}}_v$.
    \item For each (ordered) clause $(a \vee b \vee c)$, setting $N_a = a_1 \ra a_2 $, $N_b = b_1 \ra b_2 $, $N_c = c_1 \ra c_2 $, we add the following arcs: $c_2 \ra a_1$, $b_1 \ra a_2$ and $c_1 \ra a_1$. See \autoref{fig:clause}.
    \item All other arcs are forward-arcs.
\end{enumerate}

\begin{figure}[h]
    \centering
    \begin{tikzpicture}[scale=.9]
        \def\retwidth{3}
        \def\retheight{0.6cm}
        \def\Mwidth{1.7cm}
        \begin{scope}

            \node[blackvertex, label = -90:{$v$}] (u) at (0,0) {};

            \node[rectangle, draw, minimum height = \retheight, minimum width = \retwidth.cm] (Nu) at ($(u) + (.5 + \retwidth/2, 0)$) {};
            \foreach \i in {1/4} {
                    \draw ($(Nu) + (-\retwidth/2 + \i*\retwidth, -\retheight/2)$) -- ($(Nu) + (-\retwidth/2 + \i*\retwidth, \retheight/2)$);
                }
            \node[label = {[yshift = -.2cm]-90:$N^{\overline{v}}_v$}] (n1) at ($(Nu) + (-\retwidth/2 + 1/8*\retwidth, 0)$) {};
            \node[label = {[yshift = -.3cm]-90:$\cdots$}] (n2) at ($(Nu) + (-\retwidth/2 + 3/8*\retwidth, 0)$) {};
            \node (n3) at ($(Nu) + (-\retwidth/2 + 5/8*\retwidth, 0)$) {$\cdots$};

            \node[blackvertex, scale = .7] (a11) at ($(n1) + (-.35, 0)$) {};
            \node[blackvertex, scale = .7] (a12) at ($(n1) + (.15, 0)$) {};

        \end{scope}

        \begin{scope}[xshift=4.5cm]
            \node[blackvertex, label = -90:{$\ov{v}$}] (u) at (0,0) {};

            \node[rectangle, draw, minimum height = \retheight, minimum width = \retwidth.cm] (Nu) at ($(u) + (.5 + \retwidth/2, 0)$) {};
            \foreach \i in {1/4, 4/8, 3.2/4} {
                    \draw ($(Nu) + (-\retwidth/2 + \i*\retwidth, -\retheight/2)$) -- ($(Nu) + (-\retwidth/2 + \i*\retwidth, \retheight/2)$);
                }
            \node[label = {[yshift = -.2cm]-90:$N^{v}_{\ov{v}}$}] (n1) at ($(Nu) + (-\retwidth/2 + 1/8*\retwidth, 0)$) {};
            \node[label = {[yshift = -.3cm]-90:$\cdots$}] (n2) at ($(Nu) + (-\retwidth/2 + 3/8*\retwidth, 0)$) {$\cdots$};
            \node[label = {[yshift = -.2cm]-90:$N^{j}_{\ov{v}}$}] (n3) at ($(Nu) + (-\retwidth/2 + 5/8*\retwidth, 0)$) {};
            \node[label = {[yshift = -.3cm]-90:$\cdots$}] (n4) at ($(Nu) + (-\retwidth/2 + 7.4/8*\retwidth, 0)$) {$\cdots$};

            \node[blackvertex, scale = .7] (b11) at ($(n1) + (-.35, 0)$) {};
            \node[blackvertex, scale = .7] (b12) at ($(n1) + (.15, 0)$) {};

            \node[blackvertex, scale = .7] (b31) at ($(n3) + (-.2, 0)$) {};
            \node[blackvertex, scale = .7] (b32) at ($(n3) + (.25, 0)$) {};

        \end{scope}

        \begin{scope}[xshift=9.5cm]
            \node[blackvertex, label = -90:{$\ell$}] (u) at (0,0) {};

            \node[rectangle, draw, minimum height = \retheight, minimum width = 0.8*\retwidth.cm] (Nu) at ($(u) + (.5 + \retwidth/2, 0)$) {};
            \foreach \i in {1/2} {
                    \draw ($(Nu) + (-\retwidth/2 + \i*\retwidth, -\retheight/2)$) -- ($(Nu) + (-\retwidth/2 + \i*\retwidth, \retheight/2)$);
                }
            \node[label = {[yshift = -.2cm]-90:$N^{\overline{v}}_\ell$}] (n1) at ($(Nu) + (-\retwidth/2 + 1/4*\retwidth, 0)$) {};
            \node[label = {[yshift = -.2cm]-90:$N^{C}_\ell$}] (n2) at ($(Nu) + (-\retwidth/2 + 3/4*\retwidth, 0)$) {};

            \node[blackvertex, scale = .7] (c11) at ($(n1) + (-.25, 0)$) {};
            \node[blackvertex, scale = .7] (c12) at ($(n1) + (.25, 0)$) {};

            \node[blackvertex, scale = .7] (c31) at ($(n2) + (-.2, 0)$) {};
            \node[blackvertex, scale = .7] (c32) at ($(n2) + (.25, 0)$) {};

        \end{scope}
        \draw[arrow] (b11) to [bend right = 40] (a11);
        \draw[arrow] (b12) to [bend right = 40] (a12);

        \draw[arrow] (c11) to [bend right = 40] (b31);
        \draw[arrow] (c12) to [bend right = 40] (b32);

        \node at ($(u) + (-.85,0)$) {$\cdots$};

    \end{tikzpicture}
    \caption{Back-arc matchings from $N^{v}_{\ov{v}}$ to $N^{\ov{v}}_v$, where $v,\ov{v} \in \mathcal V$, and from $N^{\ov{v}}_{\ell}$ and $N^{j}_{\ov{v}}$, where $\ell$ is the $j^{th}$ occurrence of $\ov{x}$ and $x$ is the variable associated with the pair $v,\ov{v}$. Vertices are ordered as in $\prec^*$. Non-drawn arcs are all forward-arcs.}
    \label{figure:A-gadget-A-link}
\end{figure}
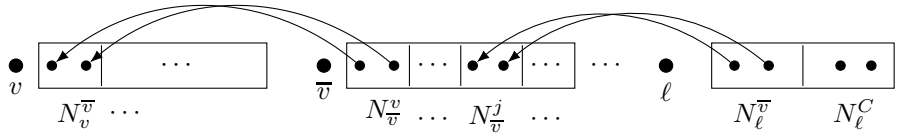

\begin{figure}[h]
    \centering
    \begin{tikzpicture}[scale=.9]
        \def\retheight{0.75cm}
        \def\Mwidth{1cm}
        \def\dist{1.2}
        \begin{scope}[xshift=-4.5cm]
            \node[blackvertex, label = -90:{$a$}] (u) at (0,0) {};

            \node[blackvertex, scale = .7, label = - 90:{$a_1$}] (a1) at ($(u) + (\dist,0)$) {};
            \node[blackvertex, scale = .7, label = - 90:{$a_2$}] (a2) at ($(a1) + (.75,0)$) {};
            %\draw[arrow, shorten >= 0pt] (a1) -- (a2);

            \node[rectangle, draw, minimum height = \retheight, fit = (a1)(a2), label = -90:{$N^{C}_{a}$}, yshift = -.1cm] (Nu) {};
        \end{scope}

        \begin{scope}
            \node[blackvertex, label = -90:{$b$}] (u) at (0,0) {};

            \node[blackvertex, scale = .7, label = - 90:{$b_1$}] (b1) at ($(u) + (\dist,0)$) {};
            \node[blackvertex, scale = .7, label = - 90:{$b_2$}] (b2) at ($(b1) + (.75,0)$) {};

            \node[rectangle, draw, minimum height = \retheight, fit = (b1)(b2), label = -90:{$N^{C}_{b}$}, yshift = -.1cm] (Nu) {};
        \end{scope}

        \begin{scope}[xshift=4.5cm]
            \node[blackvertex, label = -90:{$c$}] (u) at (0,0) {};

            \node[blackvertex, scale = .7, label = - 90:{$c_1$}] (c1) at ($(u) + (\dist,0)$) {};
            \node[blackvertex, scale = .7, label = - 90:{$c_2$}] (c2) at ($(c1) + (.75,0)$) {};

            \node[rectangle, draw, minimum height = \retheight, fit = (c1)(c2), label = -90:{$N^{C}_{c}$}, yshift = -.1cm] (Nu) {};
        \end{scope}

        \draw[arrow] (b1) to [bend right = 30] (a2);
        \draw[arrow] (c1) to [bend right = 30] (b2);
        \draw[arrow] (c2) to [bend right = 30] (a1);

    \end{tikzpicture}
    \caption{Arcs linking the gadgets associated with an (ordered) clause $C = (a \vee b \vee c)$.
        %Back-arc matchings associated with an ordered clause $C = (a \vee b \vee c)$. 
        Non-drawn arcs are all forward-arcs.}
    \label{fig:clause}
\end{figure}
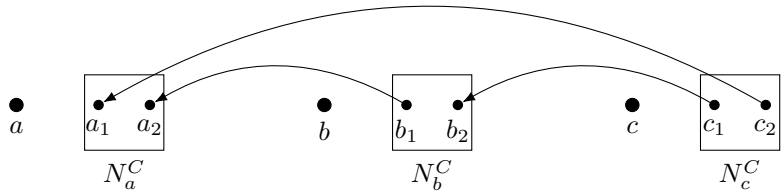

This ends the construction of the ordered tournament $(\mathcal T_{\I}, \prec^*)$. \autoref{figure:construction-of-the-tournament} illustrates the links between the gadgets associated with vertices in $\mathcal{V} \cup \mathcal{L}$.

\medskip

For every $v \in \mathcal{V}$ associated with a variable $x$, we ask that $N_{v}$ and $Y_{v}$ have as many parts of size two and five, respectively, as the number of occurrences of $\ov{x}$ as a literal plus one.
This choice is made to ensure that no two gadgets $G_u, G_{u'}$ with $u,u' \in \mathcal{V} \cup \mathcal{L} \setminus \{v\}$ have back-arc matchings with common endpoints in $G_{v}$.
For otherwise, it could be the case that every backedge graph of $T$ would contain a cycle. Here, the backedge graph induced by the union of all gadgets $G_v$ when $v \in \mc V \cup \mc L$ with respect to $\prec^*$ is a matching.

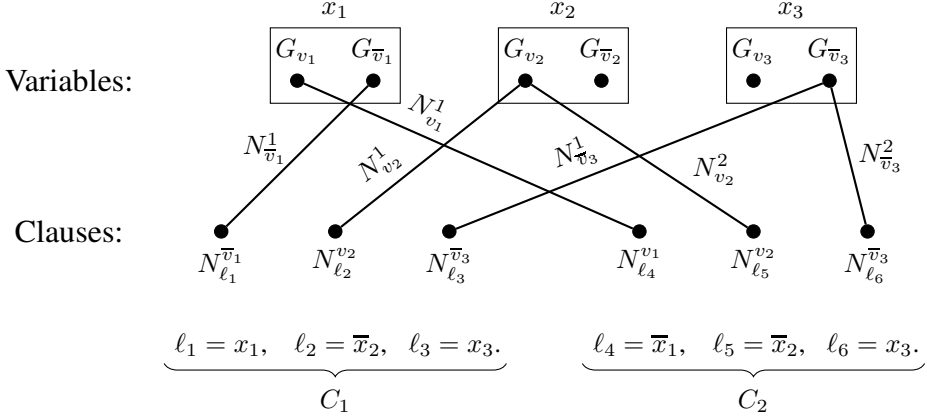
\begin{figure}[h]
    \begin{tikzpicture}
        \node (V) at (-5, 0) {\large Variables:};

        \foreach \i in {1,2,3} {
                \node[blackvertex,label=$G_{v_\i}$] (v\i) at ($(V) + (3*\i, 0)$) {};
                \node[blackvertex,label=$G_{\overline{v}_\i}$] (vb\i) at ($(v\i) + (1, 0)$) {};
                \node[draw,rectangle, fit = (v\i)(vb\i), label = $x_\i$, xscale = 1.2, minimum height = .75cm, yshift = .2cm] {};
            }

        \node (C) at (-5, -2) {\large Clauses:};

        \node[blackvertex,label={[align=center]-90:$N^{\ov{v}_1}_{\ell_1}$\\ \\$\ell_1 = x_1$}] (t1) at ($(C) + (2, 0)$) {};
        \node[blackvertex,label={[align=center,yshift = -1pt]-90:$N^{v_2}_{\ell_2}$\\[1pt] \\$\ell_2 = \overline{x}_2$}] (t2) at ($(t1) + (1.5, 0)$) {};
        \node[blackvertex,label={[align=center]-90:$N^{\ov{v}_3}_{\ell_3}$\\ \\ $\ell_3 = x_3$}] (t3) at ($(t2) + (1.5, 0)$) {};

        \node[blackvertex,label={[align=center, yshift = -1pt]-90:$N^{v_1}_{\ell_4}$\\[1pt] \\$\ell_4 = \overline{x}_1$}] (t4) at ($(t3) + (2.5, 0)$) {};
        \node[blackvertex,label={[align=center, yshift = -1pt]-90:$N^{v_2}_{\ell_5}$\\[1pt] \\$\ell_5 = \overline{x}_2$}] (t5) at ($(t4) + (1.5, 0)$) {};
        \node[blackvertex,label={[align=center]-90:$N^{\ov{v}_3}_{\ell_6}$\\ \\ $\ell_6 = x_3$}] (t6) at ($(t5) + (1.5, 0)$) {};

        \draw[-, thick] (t1) to node [pos = .5, left, yshift=.1cm] {$N^1_{\overline{v}_1}$} (vb1);
        \draw[-, thick] (t2) to node [pos = .3, sloped, above] {$N^1_{v_2}$} (v2);
        \draw[-, thick] (t3) to node [pos = .35, sloped, above] {$N^1_{\overline{v}_3}$} (vb3);
        \draw[-, thick] (t4) to node [pos = .65, sloped, above] {$N^1_{v_1}$}  (v1);
        \draw[-, thick] (t5) to node [pos = .3, right, yshift=.2cm] {$N^2_{v_2}$} (v2);
        \draw[-, thick] (t6) to node [pos = .5, right] {$N^2_{\overline{v}_3}$} (vb3);

        \draw [decorate,decoration={brace,amplitude=5pt,mirror}]
        ($(t1) + (-.25, -1.5)$) -- ($(t3) + (.25, -1.5)$) node[midway, yshift = -.5cm] {$C_1$};

        \draw [decorate,decoration={brace,amplitude=5pt,mirror}]
        ($(t4) + (-.25, -1.5)$) -- ($(t6) + (.25, -1.5)$) node[midway, yshift = -.5cm] {$C_2$};

    \end{tikzpicture}
    \caption{
        Informal description of the back-arc matchings linking the gadgets of the literals of two clauses, together with the gadgets of the variables these literals represent.
        The clauses are $C_1 = (x_1 \vee \ov{x_2} \vee x_3)$ and $C_2 = (\ov{x}_1 \vee \ov{x}_2 \vee x_3)$.
        For $i \in [3]$, each variable $x_i$ of $\mathcal{I}$ is associated with gadgets $G_{v_i}$ and $G_{\ov{v}_i}$ in $(T_\mathcal{I}, \prec^*)$.
        The edge labelled $N^{1}_{\ov{v}_1}$ represents the back-arc matching from  $N^{\ov{v}_1}_{\ell_1}$ to $G_{\ov{v}_1}$, where $\ell_1$ is the first occurrence of $x_1$.
    }

    \label{figure:construction-of-the-tournament}
\end{figure}

% NEWSTUFF
In \autoref{table:notation}, we give a summary of the notations introduced in \autoref{sec:proof}.

\begin{table}[h]
    \centering
    \begin{tabular}{%
        >{\centering\arraybackslash}m{0.25\linewidth}
        >{\raggedright\arraybackslash}m{0.65\linewidth}
        }
        \toprule
        \hyperref[sec:base-tournament]{$(T_B,\prec^*)$}
         & Base (ordered) tournament.                                                  \\
        \midrule
        \hyperref[sec:base-tournament]{$M_x$}
         & For each $x \in \mathcal{V} \cup \mathcal{L}$, vertices inducing a copy
        of the tournament of \autoref{figure:magical-tournament}.                      \\
        \midrule
        \hyperref[sec:base-tournament]{$\ell_x$}
         & For each $x \in \mathcal{V} \cup \mathcal{L}$, $\ell_x$ is the leftest vertex of $M_x$.                      \\
        \midrule
        \hyperref[sec:base-tournament]{$B_x$}
         & For $x \in \mathcal{V} \cup \mathcal{L}$, the tournament $T_B[x \cup M_x]$. \\
        \midrule
        \hyperref[sec:base-tournament]{$(B_x,\prec_x)$}
         & Block associated with $x \in \mathcal{V} \cup \mathcal{L}$.                 \\
        \midrule
        \hyperref[sec:gadgets]{$\occur(x_i$)}\;($\occur(\ov{x}_i)$)
         & The number of occurrences of the literal $x_i$ ($\ov{x}_i$), when $x_i$
        is a variable.                                                                 \\
        \midrule
        \hyperref[sec:gadgets]{$G_w$}
         & Gadget associated with vertex $w \in \mathcal{V} \cup \mathcal{L}$.         \\
        \bottomrule
    \end{tabular}
    \caption{A summary of the notation introduced in \autoref{subsec:construction}.}
    \label{table:notation}
\end{table}

In \autoref{subsec:forest-ordering_implies_True_Assignment}, we are given a forest-ordering $\prec$ of $T_{\I}$ and we need to deduce from it a truth assignment of the variables that satisfy $\mc I$.
This is done by observing the position of each vertex $x \in  \mc V \cup \mc L$ with respect to $\ell_x$.

\autoref{subsec:assignment_implies_forest_ordering}, we are given a satisfiable instance $\I$ and we want to construct a forest-ordering of $T_{\I}$ from it.
This forest-ordering will be obtained from $\prec^*$, by placing the vertices $x$ in $\mc V \cup \mc L$ on the left or on the right of $\ell_x$ depending on the truth assignment received by the variable it represents.

%--------------------------------------------------------------- 

\subsection{Building a satisfying assignment from a forest-ordering}\label{subsec:forest-ordering_implies_True_Assignment}

%--------------------------------------------------------------- 

Let $\I$ be an instance of $3$-\textsc{SAT} (as described at the beginning of \autoref{subsec:construction}) and $(T_{\I}, \prec^*)$ the tournament described in \autoref{subsec:construction}.
The goal of this subsection is to prove that if $T_{\I}$ has a forest-ordering, then $\I$ is a satisfiable  instance.

Given the forest-ordering $\prec$ of $T_{\I}$, we need to deduce from $\prec$ a truth assignment of the variables of $\I$ that satisfies $\mc I$.
This will be done using the position of each vertex $x \in  \mc V \cup \mc L$ with respect to $\ell_x$ as we explain now.

We say that $\prec$ satisfies $\Left_{\prec}(v_i)$  if $v_i \prec \ell_{v_i}$, and $\prec$ satisfies $\Right_{\prec}(v_i)$ if $\ell_{v_i} \prec v_i$. $\Left_{\prec}(v_i)$ is used to assign the  value True  to the variable  $x_i$, and $\Right_{\prec}(v_i)$ to assign the value False.

In order to get a coherent assignment for $\I$ and to be sure that each clause is satisfied, we need  to prove  the following for every forest-ordering $\prec$ of $T_{\I}$:
\begin{itemize}
    \item For every $i \in [n]$, $\prec$ satisfies either $\Left(v_i)$ and $\Right(\overline{v_i})$ or $\Right(v_i)$ and $\Left(\overline{v}_i)$.
    \item For every $i \in [n]$, $\prec$ satisfies $\Left(v_i)$ if and only if $\prec$ satisfies $\Left(\ell)$ for any literal $\ell$ that is an occurrence of $x_i$.
    \item For every $i \in [n]$, $\prec$ satisfies $\Left(\ov{v}_i)$ if and only if $\prec$ satisfies $\Left(\ell)$ for any literal $\ell$ that is an occurrence of $\ov{x}_i$.
          %\item For every $v \in \mc V$,  $\prec$ satisfies $\Left(v_i)$ if and only if $\prec$ satisfies $\Left(\ell)$ for any literal $\ell$ that is an occurrence of $x_i$. 
    \item Each clause is satisfied. With our notation, this is equivalent with proving that for each (ordered) clause $(x \vee y \vee z)$, the forest-ordering $\prec$ assigns $\Left_{\prec}(x)$ or $\Left_{\prec}(y)$ or $\Left_{\prec}(z)$.
\end{itemize}
The three first items are proved in  \autoref{lem:assignment_coherent}, and the last one in \autoref{lem:clause_sat}.
\smallskip

The next lemma is used to ensure that all the back-arcs of the back arc matchings linking pairs of gadgets appear in the backedge graph of every forest-ordering.
\begin{lemma}\label{prop:back-edgematching}
    In every forest-ordering $\prec$ of $T_{\mathcal  I}$, for every $x,z \in \mc V \cup \mc L$, such that $x \prec z$, we have $N_x \prec N_z$ and $Y_x \prec Y_z$
\end{lemma}

\begin{proof}
    Let $\prec$ be a forest-ordering of $T_{\I}$ and let $x,z  \in \mc V \cup \mc L$ such that $x \prec z$.

    Let $a \in N_x$ and $b \in N_z$, and assume for contradiction that $b \prec a$.
    By construction of $T_{\I}$ we have  $a \Ra M_x \Ra b$ (see \autoref{sec:ordering-arcs-linking-gadgets-base-tournament} and item 4. in the itemization in \autoref{section:arcs-linking-gadgets}).
    Thus, by \autoref{lem:tool}, each vertex of $M_x$ is adjacent to $a$ or $b$ in $\T_{\I}$. This implies that, in $T^{\prec}$,  $a$ or $b$ has at least two neighbours in $M_x$, and since $\T_{\I}[M_x]$ is a tree, we get that $\T_{\I}[M_x \cup \{a,b\}]$ has a cycle, contradiction.

    Now, let $a \in Y_x$ and $b \in Y_z$, and assume for contradiction that $b \prec a$. We have $a \Ra M_x \setminus \{\ell_x\} \Ra b$ and since $b \prec a$, by \autoref{lem:tool}, each vertex of $M_x \setminus \{\ell_x\}$ is adjacent to $a$ or $b$ in $\T_{\I}$. Hence, in $\T_{\I}$, one of $a$ or $b$ has two neighbours in $M_x$, and since $\T_{\I}[M_x]$ is a tree, $\T_{\I}[M_x \cup \{a,b\}]$ has a cycle, a contradiction.
\end{proof}

Recall that for every $v \in  \mc V \cup \mc L$, $B_v = \{v\} \cup M_v$

\begin{lemma} \label{prop:R}
    Let $\prec$ be a forest-ordering of $T_{\I}$.  For every $v \in \mc V \cup \mc L$, if $\prec$ satisfies $R_{\prec}(v)$, then $T_{\I}^{\prec}[B_v \cup N_v]$ is a tree.
\end{lemma}

\begin{proof}
    Let $v \in \mc V \cup \mc L$ and assume $\prec$ satisfies $R_{\prec}(v)$, i.e. $\ell_v \prec v$. Since $v \ra \ell_v$, we have that  $v\ell_v\in E(T_{\I}^{\prec})$. Moreover, $\T_{\I}[M_v]$ is a tree by Lemma~\ref{lem:unique_tree_ordering}. Hence $\T_{\I}[\{v\} \cup M_v]$ is a tree. Let $r \in V(N_v)$. By construction of $T_{\I}$ (see \autoref{sec:ordering-arcs-linking-gadgets-base-tournament}), $v \Ra N_v \Ra \ell_v$, so in particular $v \ra r \ra \ell_v$.
    If $r \prec v$, then $vr \in E(\T_{\I})$. If $v \prec r$, then $\ell_v \prec r$ and thus $\ell_vr \in E(\T_{\I})$. This implies the lemma.
\end{proof}

As announced, next lemma implies that the truth assignment of $\I$ deduced from a forest-ordering $\prec$ is coherent.

\begin{lemma}\label{lem:assignment_coherent}
    Let $\prec$ be a forest-ordering of $T_{\I}$, let $v \in \mathcal V$ and let $x \in \mc V \cup \mc L \setminus \{v\}$. If $x= \ov{v}$, or $x \in \mathcal L$ is an occurrence of $\ov{v}$, then, one of the following is satisfied:
    \begin{itemize}
        \item $L_{\prec}(v)$ and $R_{\prec}(x)$, or
        \item $R_{\prec}(v)$ and $L_{\prec}(x)$.
    \end{itemize}
\end{lemma}

\begin{proof}
    Assume $x$ is either $\ov{v}$ or is the $i^{th}$ occurrence of $\ov{v}$ for some $i\in [3k]$.

    Assume for contradiction that $R_{\prec}(v)$ and $R_{\prec}(x)$. Assume first that $x = \ov{v}$.
    By \autoref{prop:R}, both $\T_{\I}[B_v \cup N_v]$ and $\T_{\I}[B_{\ov{v}} \cup N_{\ov{v}}]$ are trees.
    By \autoref{prop:back-edgematching}, item 2. in \autoref{section:arcs-linking-gadgets}, and the assumption that $\prec$ is a forest-ordering of $T_{\I}$, there is a perfect matching between $N_v^{\ov{v}}$ and $N_{\ov{v}}^v$ in $\T_{\I}$, i.e. two edges are linking $\T_{\I}[B_v \cup N_v]$ and $\T_{\I}[B_{\ov{v}} \cup N_{\ov{v}}]$ in $\T_{\I}$, and thus $\T_{\I}$ has a cycle, a contradiction.
    If $x = \ell$  is the $i^{th}$ occurrence of $\ov{v}$, then the same reasoning holds by considering $N_v^i$ instead of $N_{\ov{v}}^v$ and $N_{\ell}^v$ instead of $N_{\ov{v}}^v$.
    \smallskip

    Assume now that $L_{\prec}(v)$ and $L_{\prec}(x)$. Assume first that $x = \ov{v}$, so we have both $v \prec \ell_v$ and $\ov{v} \prec \ell_{\ov{v}}$.

    \begin{claim}
        Let $z \in \{v, \ov{v}\}$. For every $a \in Y_z$, either $av \in E(T^{\prec})$, or $a\ell_v \in E(T^{\prec})$
    \end{claim}

    \begin{subproof}
        Recall that, by construction of $T_{\I}$, we have $Y_z \Ra z \ra  \ell_z \Ra Y_z$ (see \autoref{fig:big_block}).
        Let $a \in Y_z$. Since  $z \prec \ell_z$ and $a \ra z \ra \ell_z \ra a$, if $z \prec a$, then $za \in E(\T)$, and if $a \prec z$, then $a \prec \ell_z$ and $a\ell_z \in E(\T)$.
    \end{subproof}

    Recall that $Y_v^{\ov{v}}$ and $Y_{\ov{v}}^v$ both have $5$ vertices.
    Now consider $H$ the subgraph of $\T$ induced by $\{v, \ell_v,$ $\ov{v}, \ell_{\ov{v}}\} \cup Y_v^{\ov{v}} \cup Y_{\ov{v}}^v$.
    By \autoref{prop:back-edgematching}, there is a perfect matching between  $Y_v^{\ov{v}}$ and $Y_{\ov{v}}^v$ in $\T$ and thus in $H$. Now, together with the claim, we get that $H$ has at least $10 + 5= 15$ edges, while it has only $14$ vertices. Hence $H$ has a cycle, a contradiction.
    This proved the case where $x= \ov{v}$, in the case where $x= \ell$  is the $i^{th}$ occurrence of $\ov{v}$, the same reasoning holds
    by considering $Y_v^i$ instead of $Y_{\ov{v}}^v$ and $Y_{\ell}^v$ instead of $Y_{\ov{v}}^{v}$.
\end{proof}

The last Lemma implies that all clauses are satified by the truth assignment of $\I$ deduced from a forest-ordering of $T_{\I}$.
\begin{lemma}\label{lem:clause_sat}
    Let $\prec$ be a forest-ordering of $T_{\I}$ and let $v \in \mathcal V$ and let $(a \vee b \vee c)$ be an (ordered) clause. Then at least one of $L(x), L(y), L(z)$ holds.
\end{lemma}

\begin{proof}
    Set $N_a = a_1 \ra a_2 $, $N_b = b_1 \ra b_2 $, $N_c = c_1 \ra c_2 $. Recall that we have:
    $b_1 \ra a_2$, $c_1 \ra b_2$ and $c_2 \ra a_1$ (See \autoref{fig:clause}). By \autoref{prop:back-edgematching}, edges $b_1a_2, c_1b_2$ and $c_2a_1$ are in $E(\T_{\I})$.
    Moreover, by \autoref{prop:R}, $\T_{\I}[B_a \cup N_a]$, $\T_{\I}[B_b \cup N_b]$ and $\T_{\I}[B_c \cup N_c]$ are trees. This, together with  edges  $b_1a_2, c_1b_2$ and $c_2a_1$, implies that $\T$ has a cycle, contradiction.
\end{proof}

%------------------------------------------------------------------------------

\subsection{Building a forest-ordering from a satisfying assignment}\label{subsec:assignment_implies_forest_ordering}
For an integer $k \geq 1$, we say that an undirected graph $G$ is \emph{$k$-degenerate} if every subgraph of $G$ has a vertex of degree at most $k$. Observe that a graph is a forest if and only if it is $1$-degenerate.

This section is dedicated to the proof that, given a satisfiable instance $\mc I$ of $3$-\textsc{SAT}, $T_{\I}$ admits a forest-ordering.
Let $\nu$ be a variable assignment satisfying $\I$.
The ordered tournament $(T_{\I}, \prec^{*})$ starts with every $x \in \mathcal{V} \cup \mathcal{L}$ assigned $\Left(x)$, and we show how to use $\nu$ to decide which $x \in \mathcal{V} \cup \mathcal{L}$ should be assigned $\Right(x)$ to build a forest-ordering $\prec$ for $T_{\I}$ starting from $\prec^{*}$, as indicated in \autoref{subsec:construction}.

We set
$$\mathcal{V}_{True} = \{v_i \in \mathcal{V} \mid \nu(x_i) = True\} \cup \{\ov{v}_i \in \mathcal{V} \mid \nu(x_i) = False\}$$
%$\mathcal{V}_{True} = \{v_i \in \mathcal{V} \mid \nu(x_i) = True\} \cup \{\ov{v}_i \in \mathcal{V} \mid \nu(x_i) = False\}$, 
and $\mc V_{False} = \mc V \setminus \mc V_{True}$. We define $\mc L_{True}$ and $\mc L_{False}$ similarly.

Then we define the order $\prec$ of $T$  from $\prec^*$ by moving each vertex $x \in \mc V_{False} \cup \mc L_{False}$ right after $Y_x$.
\autoref{figure:case-vertex-to-the-right} describes the ordering $\prec$ restricted to $\{x\} \sqcup G_x \sqcup M_x\}$ when $x \in \mathcal{V}_{False} \cup  \mathcal{L}_{False}$.
When $x \in \mathcal{V}_{True} \cup  \mathcal{L}_{True}$, the ordering of $\{\{x\} \sqcup G_x \sqcup M_x\}$ is the same as $\prec^*$ and is shown in  \autoref{fig:big_block}.

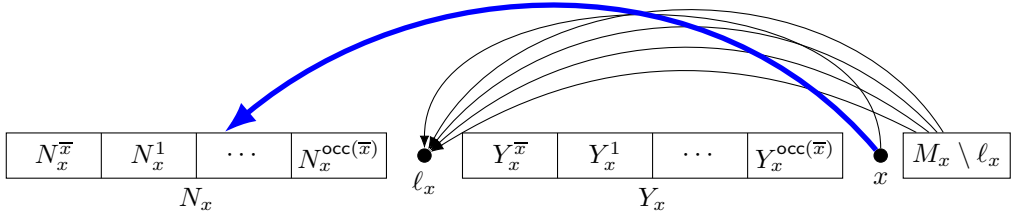
\begin{figure}[h]
    \centering
    \begin{tikzpicture}
        \begin{scope}
            \def \retwidth{4}
            \def \retheight{0.6cm}
            %\node[blackvertex, label = -90:{$x$}] (x) at (0,0) {};

            \begin{scope}
                \node[rectangle, draw, minimum width = \retwidth.cm, minimum height = \retheight, label = -90:{$N_x$}] (Rx) at (0,0) {};

                \foreach \i in {1/4,2/4,3/4} {
                        \draw ($(Rx) + (-\retwidth/2 + \i*\retwidth, -\retheight/2)$) -- ($(Rx) + (-\retwidth/2 + \i*\retwidth, \retheight/2)$);
                    }

                \node at ($(Rx) + (-\retwidth/2 + 1/8*\retwidth, 0)$) {$N^{\overline{x}}_x$};
                \node at ($(Rx) + (-\retwidth/2 + 3/8*\retwidth, 0)$) {$N^{1}_x$};
                \node at ($(Rx) + (-\retwidth/2 + 5/8*\retwidth, 0)$) {$\cdots$};
                \node at ($(Rx) + (-\retwidth/2 + 7/8*\retwidth, 0)$) {$N^{\occur(\overline{x})}_x$};
            \end{scope}

            \node[blackvertex, label = -90:{$\ell_x$}] (lx) at ($(Rx) + (0.5 + \retwidth/2, 0)$) {};

            \begin{scope}
                \node[rectangle, draw, minimum width = \retwidth.cm, minimum height = \retheight, label = -90:{$Y_x$}] (Lx) at ($(lx) + (0.5 + \retwidth/2, 0)$) {};
                \foreach \i in {1/4,2/4,3/4} {
                        \draw ($(Lx) + (-\retwidth/2 + \i*\retwidth, -\retheight/2)$) -- ($(Lx) + (-\retwidth/2 + \i*\retwidth, \retheight/2)$);

                    }
                \node at ($(Lx) + (-\retwidth/2 + 1/8*\retwidth, 0)$) {$Y^{\overline{x}}_x$};
                \node at ($(Lx) + (-\retwidth/2 + 3/8*\retwidth, 0)$) {$Y^{1}_x$};
                \node at ($(Lx) + (-\retwidth/2 + 5/8*\retwidth, 0)$) {$\cdots$};
                \node at ($(Lx) + (-\retwidth/2 + 7/8*\retwidth, 0)$) {$Y^{\occur(\overline{x})}_x$};
            \end{scope}

            \node[blackvertex, label = -90:{$x$}] (x) at ($(Lx) + (0.5 + \retwidth/2, 0)$) {};

            \node[rectangle, draw, minimum width = 1cm, minimum height = \retheight] (Mx) at ($(x) + (1, 0)$) {$M_x \setminus \ell_x$};

            \draw[arrow] (x) to [bend right = 36] (lx);
            \draw[arrow] (Mx) to [in = 60, out = 140] (lx);
            \draw[double arrow, blue] (x) to [out = 130, in = 40] (Rx);
        \end{scope}
    \end{tikzpicture}
    \caption{The figure represents a vertex $v \in \mathcal V$, $M_v$ and $G_v = N_v \sqcup Y_v$ ordered as the ordering $\prec$ defined in \autoref{subsec:assignment_implies_forest_ordering} when $v \in \mc V_{False} \cup \mc L_{False}$. Forward-arcs are not drawn. A thick arc represent all arcs in that orientation.}
    \label{figure:case-vertex-to-the-right}
\end{figure}

Hence, if $x \in \mc V_{True} \cup \mc L_{True}$,  then the only edges of $\T_{\I}[\{x\} \sqcup G_x \sqcup M_x]$ are the edges linking each vertex of $Y_x$ with $x$ (that is back-arcs in \autoref{fig:big_block}) and if $x \in \mc V_{False} \cup \mc L_{False}$, then the only edges of $\T[\{x\} \sqcup G_x \sqcup M_x]$ are the edges linking each vertex of $N_x$ with $x$ (that is back-arcs in \autoref{figure:case-vertex-to-the-right}).

We are going to prove that $\prec$ is a forest-ordering of $T_{\I}$ i.e., that $T_{\I}^{\prec}$ is a forest. More precisely, we are going to prove that $\T_{\I}$  is $1$-degenerate  by iteratively peeling off  vertices of degree at most $1$ in $\T_{\I}$.
\medskip

For every $x \in \mathcal V \cup \mathcal L$, $\T_{\I}[M_x]$ is a tree and vertices in $M_x \setminus \{\ell_x\}$ have no neighbour in $\T$ outside $M_x$. We can thus peel off every vertex of $M_x \setminus \{\ell_x\}$.

Let $x \in \mc V \cup \mc L$ and consider $\ell_x$. If $x \in \mc V_{True} \cup \mc L_{True}$, then $\ell_x$ has degree $0$ (see \autoref{fig:big_block}), and if $x \in \mc V_{False} \cup \mc L_{False}$, then the only remaining neighbour of $\ell_x$ in $\T$ is $x$ (see \autoref{figure:case-vertex-to-the-right}). We can thus also peel off $\ell_x$.

Let $T_1$ be the tournament obtained from $T$ after removing the union of all $M_x$ for $x \in \mc V \cup \mc L$.
In $\T_1$ the vertices in the following sets have degree one:
\begin{itemize}
    \item $N_x$ when $x \in \mc V_{True} \cup \mc L_{True}$
    \item $Y_x$ when $x \in \mc V_{False} \cup \mc L_{False}$
\end{itemize}

Indeed, when $x \in \mc V_{True} \cup \mc L_{True}$, vertices in $N_x$ have degree $0$ in $\T[B_x \cup M_x]$ (see Figures~\ref{fig:big_block} and \ref{fig:big_block_for_literal}), and thus their only neighbours are through the back arcs of the back-arc matching linking the gadgets.
Similarly, when $x \in \mc V_{False} \cup \mc L_{False}$, vertices in $Y_x$ have degree $0$ in $\T[B_x \cup M_x]$ (see Figures~\ref{figure:case-vertex-to-the-right} and~\ref{figure:case-vertex-to-the-right-literal}), and thus their only neighbours are through the back arcs of the back-arc matching linking the gadgets.

\begin{figure}[h]
    \centering
    \begin{tikzpicture}
        \begin{scope}
            \def \retwidth{2}
            \def \retheight{0.6cm}

            \begin{scope}
                \node[rectangle, draw, minimum width = \retwidth.cm, minimum height = \retheight, label = -90:{$N_z$}] (Rx) at (0,0) {};

                \foreach \i in {1/2} {
                        \draw ($(Rx) + (-\retwidth/2 + \i*\retwidth, -\retheight/2)$) -- ($(Rx) + (-\retwidth/2 + \i*\retwidth, \retheight/2)$);
                    }

                \node at ($(Rx) + (-\retwidth/2 + 1/4*\retwidth, 0)$) {$N^{\overline{v}_i}_z$};
                \node at ($(Rx) + (-\retwidth/2 + 3/4*\retwidth, 0)$) {$N^{C}_z$};
            \end{scope}

            \node[blackvertex, label = -90:{$\ell_z$}] (lx) at ($(Rx) + (0.5 + \retwidth/2, 0)$) {};

            \begin{scope}
                \node[rectangle, draw, minimum width = \retwidth.cm, minimum height = \retheight, label = -90:{$Y_z$}] (Lx) at ($(lx) + (0.5 + \retwidth/2, 0)$) {};
                \foreach \i in {1/6, 2/6, 3/6, 4/6, 5/6}{
                \node[blackvertex, scale = .7] (y\i) at ($(Lx) + (\i*\retwidth - \retwidth/2, 0)$ ) {};
                }

            \end{scope}

            \node[blackvertex, label = -90:{$z$}] (x) at ($(Lx) + (0.75+\retwidth/2, 0)$) {};

            \node[rectangle, draw, minimum width = 1cm, minimum height = \retheight] (Mx) at ($(x) + (0.5 + \retwidth/2, 0)$) {$M_z \setminus \ell_z$};

            \draw[arrow] (Mx) to [bend right = 40] (lx);
            \draw[arrow] (x) to [bend right = 40] (lx);
            \draw[double arrow, blue] (x) to [bend right = 40] (Rx);
        \end{scope}
    \end{tikzpicture}
    \caption{The figure represents a vertex $z \in \mathcal V$, $M_z$ and $G_z = N_z \sqcup Y_z$ ordered as the ordering $\prec$ defined in \autoref{subsec:assignment_implies_forest_ordering}. Forward-arcs are not drawn. A thick arc represent all arcs in that orientation.}
    \label{figure:case-vertex-to-the-right-literal}
\end{figure}
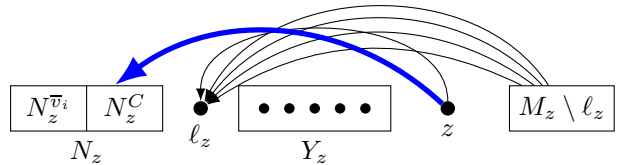

For $v \in \mathcal{V} \cup \mathcal{L}$, we may refer to $Y_v$ as the \emph{$Y$-gadget} of $v$.
Let $T_2$ be the tournament obtained after removing the above sets of vertices. In $\T_2$ the vertices in the following sets have degree one:
\begin{itemize}
    \item $Y_x$ when $x \in \mc V_{True} \cup \mc L_{True}$ (their only neighbours is $x$)
    \item $N_v$ for $v \in \mc V_{False}$.
    \item $N_{\ell}^{var(\ell)}$ for $\ell \in \mc L_{False}$, where $var(\ell)$ is the variable corresponding to $\ell$.
\end{itemize}
The first item holds because the neighbours of vertices in $Y_x$ are $x$ and the end-vertices of the back arcs of the back-arc matchings linking them with the $Y$-gadgets associated with (some of the) vertices in $\mc V_{False} \cup \mc L_{False}$, which already have been removed.
The two other items hold for the similar reasons.

After deleting all these sets, the vertices in $\mc V \cup \mc L_{True}$ have degree $1$, we can thus remove them. We call $T_3$ the tournament induced by the remaining vertices.

$T_3$ has the following vertices:
$\mc L_{False}$ and for each $\ell \in \mc L_{False}$, $N_{\ell}^{C_{\ell}}$ where $C_{\ell}$ is the clause containing $\ell$.
Let $C=(x,y,z)$ be a clause. At least one of $x,y,z$ is in $\mc L_{True}$, say $z$. If both $x$ and $y$ are False, we have $\T_3[\{x,y\} \cup N_x^{C} \cup N_y^{C}]$ is a path (see \autoref{fig:clause-false-false-true}) and a connected component of $\T_3$ and thus we are done. If $x$ or $y$ are True, the same reasoning holds.

\begin{figure}[hbt!]
    \centering
    \begin{tikzpicture}[scale=.9]
        \def\retheight{0.75cm}
        \def\Mwidth{1cm}
        \def\dist{1.2}
        \begin{scope}[xshift=-4.5cm]

            \node[blackvertex, scale = .7, label = - 90:{$x_1$}] (a1) at (0,0) {};
            \node[blackvertex, scale = .7, label = - 90:{$x_2$}] (a2) at ($(a1) + (.75,0)$) {};

            \node[rectangle, draw, minimum height = \retheight, fit = (a1)(a2), label = -90:{$N^{C}_{x}$}, yshift = -.1cm] (Nu) {};

            \node[blackvertex, label = -90:{$x$}] (u) at ($(Nu) + (0.5+\dist, 0)$) {};
            \draw[arrow] (u) to [bend right = 80] (a1);
            \draw[arrow] (u) to [bend right = 80] (a2);
        \end{scope}

        \begin{scope}

            \node[blackvertex, scale = .7, label = - 90:{$y_1$}] (b1) at (0,0) {};
            \node[blackvertex, scale = .7, label = - 90:{$y_2$}] (b2) at ($(b1) + (.75,0)$) {};

            \node[rectangle, draw, minimum height = \retheight, fit = (b1)(b2), label = -90:{$N^{C}_{y}$}, yshift = -.1cm] (Nu) {};

            \node[blackvertex, label = -90:{$y$}] (u) at ($(Nu) + (0.5+\dist, 0)$) {};
            \draw[arrow] (u) to [bend right = 80] (b1);
            \draw[arrow] (u) to [bend right = 80] (b2);
        \end{scope}

        \begin{scope}[xshift=4.5cm]
            \node[blackvertex, label = -90:{$z$}] (u) at (0,0) {};

            \node[blackvertex, scale = .7, label = - 90:{$z_1$}] (c1) at ($(u) + (\dist,0)$) {};
            \node[blackvertex, scale = .7, label = - 90:{$z_2$}] (c2) at ($(c1) + (.75,0)$) {};

            \node[rectangle, draw, minimum height = \retheight, fit = (c1)(c2), label = -90:{$N^{C}_{z}$}, yshift = -.1cm] (Nu) {};
        \end{scope}

        \draw[arrow] (b1) to [bend right = 30] (a2);
        \draw[arrow] (c1) to [bend right = 30] (b2);
        \draw[arrow] (c2) to [bend right = 30] (a1);

    \end{tikzpicture}
    \caption{Back-arcs associated with ordering $\prec$ defined in \autoref{subsec:assignment_implies_forest_ordering} with an ordered clause $C = (x,y,z)$ in $T_3$, where $x,y \in  \mathcal{L}_{False}$ and $z \in \mathcal{L}_{True}$.}
    \label{fig:clause-false-false-true}
\end{figure}
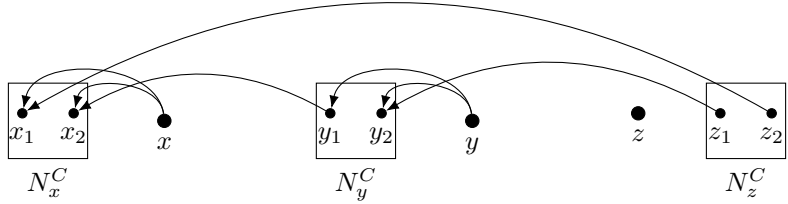

%%%%%%%%%%%%%%%%%%%%%%%%%%%%%%%%%%%%%%%%%%%%%%%%

\vspace{-0.2cm}

\section{Conclusion}

We think the $\mathcal C$-\textsc{FAS} Problem is interesting, and quite natural, in particular because of its links with the dichromatic number and the clique number of tournaments. It is also a way to define classes of tournament (given an undirected class of graphs $\mc C$, a $\mc C$-tournaments is a tournament $T$ that admits $\mc C$-FAS) that has been often used recently (see for example Section 4 in~\cite{NGUYEN2025146} where many such classes are looked at).
\medskip

As explained in the introduction, when $\mathcal C_k$ is the class of graphs with clique number at most $k$, the $\mathcal C_k$-\textsc{FAS} problem has been recently proved to be $NP$-complete when $k \geq 4$, is easily seen to be polynomial when $k \leq 2$, but its complexity is open when $k=3$.

\begin{problem}[\cite{A24}]
Given a tournament $T$, what is the complexity of deciding if $T$ has a triangle-free FAS?
\end{problem}

On the same flavour, Aboulker et al.\cite{AACT24} proved that an approximation version of the above problem is polynomial.
More precisely, they prove that there is a constant $c$ and a polynomial-time algorithm that, given a tournament $T$, correctly concludes that $\diomega(T) \geq 3$ or finds an order $\prec$ of $V(T)$ such that $\omega(T^\prec) \leq c$.
It is thus natural to conjecture the following.

%such that one can decide in poly-time if, given a tournament $T$, $\diomega(T) \geq 3$, or $\diomega(T) \leq c$. 

\begin{conjecture}
    There is a function $f$ such that for every integer $k$, there is a polynomial-time algorithm that, given a tournament $T$, correctly concludes that $\diomega(T) \geq k$, or finds an order $\prec$ of $V(T)$ such that $\omega(T^\prec) \leq f(k)$
\end{conjecture}

In~\cite{AACL23}, it is proved that a tournament has a forest-ordering if and only if it has a tree-ordering. (For the curious reader, the proof is easy:  simply start from a forest-ordering, runs through vertices from left to right, and when two consecutive vertices $x \prec y$ are in distinct connected component, just switch the ordering of $x$ and $y$). Together with \autoref{thm:main}, it implies the following:
\begin{theorem}
    The $\mc C$-\textsc{FAS} problem is NP-complete when $\mc C$ is the set of all trees.
\end{theorem}

It is natural to ask for which class of graphs $\mc C$, the $\mc C$-\textsc{FAS} Problem is in $P$. Two natural candidates of such $\mc C$ are paths and matching. More formally:

\begin{problem}
What is the complexity of the $\mc C$-\textsc{FAS} Problem when $\mc C$ is the set of all paths? when $\mc C$ is the set of graphs with maximum degree $1$?
\end{problem}

\clearpage

\section{Code for \autoref{lem:unique_tree_ordering}}\label{section:magical_code}

The following code is used to verify that the tournament shown in \autoref{figure:magical-tournament} admits an unique tree-ordering.

\begin{python}
    from itertools import permutations

    # is_forest(T,P)
    # Input:
    # - a tournament, given as a 8x8 matrix `T` such that
    #   T[u][v] iff there is an arc u -> v
    # - an ordering of V(T), given as a permutation `P`
    #   of [0,7] such that u < v iff P[u] < P[v]
    # Output:
    # - a boolean: is `T^<` a forest?
    def is_forest(T, P):
    visited = set()

    #   Recursive depth-first search looking for a cycle
    #   in T^< starting from rot `u`, and assuming we
    #   just visited vertex `parent` (initially, u itself).
    def dfs(u, parent):
    if u in visited:
    return True
    visited.add(u)
    return any(T[u][v] == (P[v] <= P[u]) and dfs(v, u)
    for v in range(8) if v != parent)

    return all(x in visited or not dfs(x, x)
    for x in range(8))

    # The tournament in Figure 1.
    T = [
    [ 0, 1, 1, 0, 0, 0, 1, 0],
    [ 0, 0, 1, 1, 1, 1, 0, 0],
    [ 0, 0, 0, 1, 1, 0, 1, 1],
    [ 1, 0, 0, 0, 1, 1, 1, 1],
    [ 1, 0, 0, 0, 0, 1, 1, 1],
    [ 1, 0, 1, 0, 0, 0, 1, 1],
    [ 0, 1, 0, 0, 0, 0, 0, 1],
    [ 1, 1, 0, 0, 0, 0, 0, 0]
    ]

    # We iterate over all permutations of [0,7], and check
    # that the only permutation such that is_forest(T, P)
    # returns True is the identity.
    for P in permutations(range(8)):
    assert is_forest(T, P) == (P == (0,1,2,3,4,5,6,7))
\end{python}

\clearpage

\bibliographystyle{abbrvnat}
% use the following instead if you encounter problems 
%\bibliographystyle{alpha}
\bibliography{main}
\label{sec:biblio}

\end{document}